\documentclass[11pt]{amsart}
\usepackage{amssymb,latexsym,graphicx, amscd, enumitem}
\newtheorem{theorem}{Theorem}[section]
\newtheorem{prop}[theorem]{Proposition}
\newtheorem{lemma}[theorem]{Lemma}
\newtheorem{remark}[theorem]{Remark}
\newtheorem{question}[theorem]{Question}
\newtheorem{definition}[theorem]{Definition}

\newtheorem{example}[theorem]{Example}

\DeclareMathOperator{\ep}{\varepsilon}
\DeclareMathOperator{\C}{\mathbb C}
\DeclareMathOperator{\R}{\mathbb R}

\DeclareMathOperator{\N}{\mathbb{N}}

\newcommand{\inmult}{\, \lrcorner \,}
\begin{document}

\title[$J$-holomorphic curves from closed $J$-anti-invariant forms]{$J$-holomorphic curves from closed $J$-anti-invariant forms}


\author{Louis Bonthrone}
\address{Mathematics Institute\\  University of Warwick\\ Coventry, CV4 7AL, England}
\email{L.Bonthrone@warwick.ac.uk}

\author{Weiyi Zhang}
\address{Mathematics Institute\\  University of Warwick\\ Coventry, CV4 7AL, England}
\email{Weiyi.Zhang@warwick.ac.uk}

\begin{abstract} We study the relation between $J$-anti-invariant $2$-forms and pseudoholomorphic curves in this paper.  We show the zero set of a closed $J$-anti-invariant $2$-form on an almost complex $4$-manifold supports a $J$-holomorphic subvariety in the canonical class. This confirms a conjecture of Draghici-Li-Zhang. A higher dimensional analogue is established.  We also show the dimension of closed $J$-anti-invariant $2$-forms on an almost complex $4$-manifold is a birational invariant, in the sense that it is invariant under degree one pseudoholomorphic maps. 
\end{abstract}
 \maketitle

\tableofcontents
\section{Introduction}
Since the 1980s there has been a well-known folklore theorem (see \cite{LeB, Hon}) which says that for a generic Riemannian metric on a $4$-manifold with positive self-dual second Betti number, the zero set of a self-dual harmonic $2$-form is a finite number of embedded circles. It is the starting point of Taubes' attempts, {\it e.g.} \cite{Tsd}, to generalise the identification of Seiberg-Witten invariants and Gromov invariants for symplectic $4$-manifolds to general compact oriented $4$-manifolds. 

Following the philosophy of \cite{ZintJ}, any statement for smooth maps between smooth manifolds in terms of  Thom's transversality should also have its counterpart in the pseudoholomorphic setting without requiring the transversality or genericity, but using the notion of pseudoholomorphic subvarieties. 
Hence, the above genericity statement for the zero set of a self-dual harmonic $2$-form in the smooth category should find its counterpart in the almost complex setting without assuming genericity. It is stated as Question 1.6 in \cite{ZintJ} which first appeared in \cite{DLZiccm}. We recall it in the following. 

Let $(M^{2n}, J)$ be an almost complex manifold. The almost complex
structure acts on the bundle of real $2$-forms $\Lambda^2$ as the following
involution, $\alpha(\cdot, \cdot) \rightarrow \alpha(J\cdot,
J\cdot)$. This involution induces the splitting
\begin{equation} \label{formtypeJ}
\Lambda^2=\Lambda_J^+\oplus \Lambda_J^-
\end{equation}
corresponding to the eigenspaces of eigenvalues $\pm 1$ respectively.  The sections are called  $J$-invariant and
$J$-anti-invariant 2-forms respectively. The spaces of these sections are denoted by $\Omega^{\pm}_J$. The bundle  $\Lambda^-_J$ inherits an almost complex structure, still denoted by $J$, from $J\alpha(X, Y)=-\alpha(JX, Y)$.

On the other hand, for any Riemannian metric $g$ on a $4$-manifold, we have the well-known self-dual, anti-self-dual splitting of the bundle of $2$-forms, 
\begin{equation} \label{formtypeg}
\Lambda^2=\Lambda_g^+\oplus \Lambda_g^-.
\end{equation}
 When $g$ is compatible with $J$, {\it i.e.} $g(Ju, Jv)=g(u, v)$, we have $\Lambda_J^-\subset \Lambda_g^+$. In particular, it follows that a closed $J$-anti-invariant $2$-form is a $g$-self-dual harmonic form. Hence, a closed $J$-anti-invariant $2$-form is the natural almost complex refinement of a self-dual harmonic form on an almost complex $4$-manifold. Thus, our expectation is that the zero set of a $J$-anti-invariant $2$-form is a $J$-holomorphic curve.

Since the complex line bundle $\Lambda_J^-$ can be viewed as a natural generalisation of the canonical bundle of a complex manifold, it is instructive to take a brief digression and consider what is known in the complex setting.  On a complex surface, if $\alpha$ is a closed $J$-anti-invariant $2$-form then $J\alpha$ is also closed and $\alpha+iJ\alpha$ is a holomorphic $(2, 0)$ form. Hence the zero set $\alpha^{-1}(0)$ is a canonical divisor of $(M, J)$, {\it e.g.} by the Poincar\'e-Lelong theorem. This meets our expectations in the case when the almost complex structure is integrable. 

In the past $10$ years, there are many constructions of non-integrable almost complex structures that admit non-trivial closed $J$-anti-invariant $2$-forms, {\it e.g.} \cite{DLZ2, FT} and many others. In all of these examples, the zero sets are all $J$-holomorphic subvarieties.  For example, these forms on K3 or $T^4$ constructed in \cite{DLZ2} are nowhere vanishing.

In this paper we are able to confirm our above speculation on the zero set of a $J$-anti-invariant $2$-form for any compact almost complex $4$-manifold.

\begin{theorem}\label{4ZJhol}
Suppose $(M, J)$ is a compact connected almost complex $4$-manifold and $\alpha$ is a non-trivial closed $J$-anti-invariant $2$-form. Then the zero set $Z$ of $\alpha$ supports a $J$-holomorphic $1$-subvariety $\Theta_{\alpha}$ in the canonical class $K_J$.
\end{theorem}

We will call the $J$-holomorphic $1$-subvariety $\Theta_{\alpha}$ stated in theorem the {\it zero divisor} of $\alpha$.

Here, a closed set $C\subset M$ with finite, non-zero
$2$-dimensional Hausdorff measure is said to be an irreducible $J$-holomorphic $1$-subvariety \cite{T1} if it has no isolated points and if the complement of a finite set
of points in $C$ is a connected smooth submanifold with $J$-invariant tangent space. A $J$-holomorphic $1$-subvariety is a finite set of pairs $\{(C_i, m_i), 1\le i\le m < \infty\}$, where each $C_i$ is an irreducible $J$-holomorphic $1$-subvariety and each $m_i$ is a positive integer.

The general scheme to prove Theorem \ref{4ZJhol} is similar to what is used in \cite{ZintJ} where we prove that the intersection of a compact $4$-dimensional pseudoholomorphic subvariety and a compact almost complex submanifold of codimension $2$ in a (not necessarily compact) almost complex manifold is a pseudoholmorphic $1$-subvariety. This basic strategy traces back to \cite{King} at least, where it works in the complex analytic setting. In the pseudo-holomorphic situation, this strategy was worked out by Taubes \cite{T}.

More concretely, the plan is to first show that $Z$ has finite 2-dimensional Hausdorff measure; this is done in section 2. The idea is to foliate neighbourhoods of points in $Z$ by $J$-holomorphic disks. Applying a dimension reduction argument with the help of a unique continuation result (Proposition \ref{noopen}) we are able to reduce our study to the intersection of $Z$ with $J$-holomorphic disks. We establish the positivity of such intersections in Lemma \ref{intJdisk2} by 
exhibiting a holomorphic trivialisation of $\Lambda_J^-$ over a given $J$-holomorphic disk. 
This lemma is the counterpart of Gromov's positivity of intersections of $J$-holomorphic curves with complex submanifolds of real codimension two \cite{Gr}. 

 If in addition we can find a ``positive cohomology assignment" for $Z$ in the sense of Taubes, which plays the role of intersection number of the set $Z$ with each local disk, we are able to show that $Z$ is a $J$-holomorphic $1$-subvariety by Proposition 6.1 of \cite{T} (stated as Proposition \ref{pcaholo} in our paper). 

Our strategy to associate a positive cohomology assignment to $Z$ is to view $J$-anti-invariant $2$-forms as sections of the bundle $\Lambda_J^-$. Now a $J$-anti-invariant form $\alpha$ defines a $4$-dimensional submanifold $\Gamma_{\alpha}$ in the total space of $\Lambda_J^-$ whose intersection with $M$, as submanifolds of $\Lambda_J^-$, describe the zero set of the form. Given a disk in $M$ whose boundary does not intersect $\Gamma_{\alpha}$, we can compose with a section and perturb to obtain a disk $\sigma ' : D \rightarrow \Lambda_J^-$ which intersects $M$ transversely. Then the oriented intersection number of $\sigma'$ defines a positive cohomology assignment. A finer study of positive cohomology assignment also gives rise to the desired information for the homology class of the zero divisor. 

Theorem \ref{4ZJhol} could be extended to sections of the bundle $\Lambda_{\R}^{n, 0}$ of real parts of $(n, 0)$ forms, which has a natural complex line bundle structure induced by the almost complex structure on $M$. The space of its sections is denoted $\Omega_{\R}^{n, 0}$. We have Theorem \ref{highdim} which, if Question \ref{conj} is answered affirmatively,  says that the zero set of a non-trivial closed form in $\Omega_{\R}^{n, 0}$ supports a pseudoholomorphic subvariety of real  codimension $2$. The key to establish this result is again a version of Lemma \ref{intJdisk2} for the bundle $\Lambda_{\R}^{n, 0}$; this is our Lemma \ref{highJdisk2}.

In Section \ref{bir} we study the relation between $J$-anti-invariant forms and birational geometry of almost complex manifolds. Recall that we have the cohomology groups \cite{LZ}  $$H_J^{\pm}(M)=\{ \mathfrak{a} \in H^2(M;\mathbb R) | \exists \; \alpha\in \mathcal Z_J^{\pm} \mbox{ such that } [\alpha] = \mathfrak{a} \}$$ generalising the real Hodge cohomology groups, where $\mathcal Z_J^{\pm}$ are the spaces of closed $2$-forms in $\Omega_J^{\pm}$.  It is proven in \cite{DLZ} that $H_J^+(M)\oplus H_J^-(M)=H^2(M; \R)$ when $\dim_{\R} M=4$. The dimensions  of  the vector spaces $H_J^{\pm}(M)$ are denoted by $h_J^{\pm}(M)$. 

In \cite{ZintJ} it is shown that the natural candidates which generalise birational morphisms to the almost complex category are degree one pseudoholomorphic maps. Using the local model given by Lemma \ref{intJdisk2} together with the foliation-by-disks technique as used to establish Theorem \ref{4ZJhol}, 
one can study the extension properties of closed $J$-holomorphic disks. This gives  us Proposition \ref{Hartog}, which should be compared with the Hartogs extension for pseudoholomorphic bundles over almost complex $4$-manifolds established in \cite{CZ}. 

With this Hartogs type extension for closed $J$-anti-invariant $2$-forms in hand, we are able to show that the dimension of $J$-anti-invariant cohomology is a birational invariant.

\begin{theorem} \label{birationalintro}
Let $\psi : (M_1,J_1) \rightarrow (M_2,J_2)$ be a degree $1$ pseudoholomorphic map between closed, connected almost complex $4$-manifolds. Then $h_{J_1}^-(M_1)=h_{J_2}^-(M_2)$.
\end{theorem}

Together with the almost complex birational invariants defined in  \cite{CZ}, including plurigenera, Kodaira dimension, and irregularity, we have a rich source of invariants to study the birational geometry of almost complex manifolds.

In the last section, two applications are discussed. In Section \ref{mul}, we provide a definition of multiplicity of zeros of a continuous function $u:  D^2 \rightarrow \mathbb R^2$ which generalises the multiplicity of zeros of a holomorphic function.  This subsection could also be viewed as some explicit calculations of the intersection number used throughout the paper. In Section \ref{SW}, we discuss the relation between Theorem \ref{4ZJhol} and Taubes' program of SW=Gr for $4$-manifolds.

This work is partially supported by EPSRC grant EP/N002601/1. We are indebted to Tedi Draghici for pointing out that the divisor-to-section correspondence does not hold for  $J$-anti-invariant forms and their zero divisors, even when $J$ is tamed, by the examples in \cite{DLZ2}. We would like to thank Tian-Jun Li for his interest and his suggestions to improve the presentation of the paper. We are grateful to an anonymous referee for meticulous reading and helpful suggestions which improve the paper greatly, and another referee for pointing out reference \cite{Biq}.

\section{Finite $2$-dimensional Hausdorff measure}\label{HD2}
In this section, we assume $M$ is a $4$-dimensional closed manifold.  The peculiarity of dimension $4$ is that the Hodge operator $*_g$ of a Riemannian metric $g$ on $M$ also acts as an involution on $\Lambda^2$. Thus we have the self-dual, anti-self-dual splitting of the bundle of $2$-forms
$$\Lambda^2=\Lambda_g^+\oplus\Lambda_g^-.$$

On the other hand given an almost complex structure $J$ on $M$, we also get a splitting of the bundle of $2$-forms into $J$-invariant and $J$-anti-invariant forms
\begin{equation*}
\Lambda^2=\Lambda_J^+\oplus \Lambda_J^-.
\end{equation*}
The spaces of their sections are denoted $\Omega^{\pm}_J$. 
Moreover, we can always choose a compatible $g$ in the sense that $g$ is $J$-invariant, {\it i.e.} $g(Ju, Jv)=g(u, v)$. The pair $(g, J)$ defines a $J$-invariant (in general, non-closed) $2$-form $\omega$ by $$\omega(u, v)=g(Ju, v).$$ Such a triple $(J, g, \omega)$ is called an almost Hermitian structure. We have the decompositions 
\begin{equation}
\Lambda_g^+=\underline{\mathbb R}(\omega)\oplus \Lambda_J^-,
\end{equation}
\begin{equation}
\Lambda_J^+=\underline{\mathbb R}(\omega)\oplus \Lambda_g^-.
\end{equation}
In particular, $\Lambda_J^-\subset \Lambda_g^+$ and it follows that  every closed $J$-anti-invariant form is a harmonic $g$-self-dual form, {\it e.g.} Lemma 2.6 of \cite{DLZ}. 

Also recall that $\Lambda_J^-$ inherits an almost complex structure, still denoted by $J$, from $J\beta(X, Y)=-\beta(JX, Y)$.

In this section, we will show that the $2$-dimensional Hausdorff measure of the zero locus, $Z$, of any closed $J$-anti-invariant $2$-form is finite. 

\begin{prop} \label{hausdorff}
Let $(M,J)$ be a closed, connected, almost complex 4-manifold and suppose that $\alpha$ is a non-trivial, closed, $J$-anti-invariant $2$-form. Then the zero set $Z$ of $\alpha$ is compact, with Hausdorff dimension $2$ and finite $2$-dimensional Hausdorff measure. 
\end{prop}
\begin{remark}
Since every closed $J$-anti-invariant form is a harmonic $g$-self-dual form for a compatible $g$, it follows from \cite{bar} that the zero locus $Z=\alpha^{-1}(0)$ is a countably $2$-rectifiable set. Recall that a subset of an $n$-dimensional Riemannian manifold $M$ is called countably $k$-rectifiable if it can be written as a countable union of sets of the form $\phi(X)$, where $X\subset \mathbb R^k$ is bounded and $\phi: X\rightarrow M$ is a Lipschitz map. However, it is not clear whether such a set would have finite $2$-dimensional Hausdorff measure by \cite{bar}. 
\end{remark}

Considering $\alpha$ as a smooth section of the bundle $\Lambda^-_J$ the compactness of $Z$ follows immediately from the continuity of $\alpha$ since a closed subset of a compact space is compact. Hence we can cover $Z$ by finitely many balls $B_i$. To prove Proposition \ref{hausdorff}, we need to show that $C\epsilon^{-2}$ many $\epsilon$-balls will be enough to cover each $B_i$. These balls $B_i$ may be taken small enough such that they are foliated by $J$-holomorphic disks as recalled in the following.

\subsection{Coordinates charts provided by disks foliation}

Fixing $x \in M$, we can find a neighbourhood $U$ of $x$ and a non-degenerate 2-form $\Omega$ on $U$ such that $J$ is compatible with $\Omega$ in $U$. This pair $(\Omega,J)$ induce an almost Hermitian metric on $U$. Now we can identify a geodesic ball centred at $x$ with a ball in $\R^4$ centred at the origin. Identify $\R^4 = \C^2$ such that
\begin{equation*}
\Omega_x = \omega_0 = dx^1 \wedge dx^2 + dx^3 \wedge dx^4 = \frac{i}{2} \left( dw^0 \wedge d\bar{w}^0 + dw^1 \wedge d\bar{w}^1 \right), 
\end{equation*}
with complex coordinates $(w^0,w^1)=(x^1, x^2, x^3, x^4)$. Such coordinates wil be called Gaussian coordinates centred at $x$. 
So we may as well assume that $J$ is an almost complex structure on $\C^2$ which agrees with the standard complex structure $J_0$ at the origin.

Let $D\subset\mathbb C$ be the open disk of radius $\rho$ and $D_w := \{ (\xi,w) | |\xi| < \rho \}\subset \mathbb C^2$, where $w \in D$. We will write $\bar D_w$ (resp. $\bar D$) for the closure of $D_w$ in $\mathbb C^2$ (resp. $D$ in $\mathbb C$). Now Lemma 5.4 of \cite{T} yields a diffeomorphism $Q : D \times D \rightarrow \C^2$, such that 
\begin{itemize}
\item $\forall w \in D$, $Q(D_w)$ is a $J$-holomorphic submanifold containing $(0,w)$, 
\item $\forall w \in D$, there exists $z$ depending only on $\Omega$ and $J$ such that
\begin{equation*}
|(\xi,w) - Q(\xi,w)| \leq z \cdot \rho \cdot |\xi|,
\end{equation*}
\item $\forall w \in D$, the derivatives of order $m$ of $Q$ are bounded by $z_m \cdot \rho$, where $z_m$ depends only on $\Omega$ and $J$.
\end{itemize}
Such diffeomorphisms shall be called $J$-fibre-diffeomorphisms. It is important to remark that we can change the direction of these disks by rotating the original Gaussian coordinate chart chosen. More precisely given $\kappa \in \mathbb CP^1$ we can choose $Q$ such that $Q(D_0)$ is tangent at the origin to the line determined by $\kappa$.

Let $u:D \rightarrow M$ be an embedded $J$-holomorphic disk with $x=u(0)$. We can further choose the coordinate system such that the almost complex structure $J$ behaves particularly well along the image $u(D)$, specified below. This is essentially a reformulation of the construction on page 903 of \cite{T}. This will be used in Lemma \ref{intJdisk2} below.

Let $(\xi,w)$ be the coordinates associated with the above $Q$. Since the disks of constant $w$ are $J$-holomorphic the almost complex structure $J$ must decompose, with respect to the splitting $T(D \times D) = TD \oplus TD = \R^2 \oplus \R^2$, as follows:
\begin{equation*} 
J = \left( \begin{array}{cc}
a & b \\
0 & a'
\end{array} \right)
\end{equation*}
Here $a,a',b$ are $2 \times 2$ matrix valued functions on $D \times D$ such that the condition $J^2 = - I$ is satisfied.

We can further choose coordinates $(\xi_1, w_1)$ such that $u(D)$ is just $\xi_1=0$, at least locally near $x$. Indeed as remarked previously we can choose the direction of the foliation such that $Q(D_0)$ intersects $u(D)$ transversally at $u(0)$. The transversality condition facilitates the application of the implicit function theorem to find, after shrinking $D$ if necessary, a smooth map $\tau : D \rightarrow \R^2$ such that $\tau(0)=0$ and $u(w)=Q(\tau(w),w)$. We let $(\xi_1, w_1) =(\xi - \tau(w),w)$. Thus, in the $(\xi_1, w_1)$ coordinates, the matrix $b$ obeys $b(0, \cdot)=0$.  
We can make a further change to coordinates $(\xi_2,w_2):=(g_1(\xi_1, w_1)\cdot \xi_1, g_2(w_1))$, for suitable smooth matrix value functions $g_1$ and $g_2$ such that,  in addition to the general requirement $J^2=-I$, we have

\begin{equation*}
a\equiv \left( \begin{array}{cc}
0 & 1 \\
-1 & 0 
\end{array} \right) \hbox{ and }\,\, 
a'(0, \cdot)\equiv \left( \begin{array}{cc}
0 & 1 \\
-1 & 0 
\end{array} \right).
\end{equation*}

To summarize, the discussion above allows us to take coordinates in a neighbourhood of $u(0)$ such that $J=J_0$ along $u(D)$. Later, we will denote such coordinates, $(w_2,\xi_2)$, by $(x^1, x^2, x^3, x^4)$ so that $u(D)$ is described by $x^3\!=\!x^4\!=\!0$ near $u(0)$.

\subsection{$J$-anti-invariant $2$-forms}
We continue assuming $u:D \rightarrow M$ is an embedded $J$-holomorphic disk and $U$ is a neighbourhood of $u(0)$ with the coordinates described above.  On $u(D) \cap U$ define 
\begin{equation} \label{phi}
\phi_0|_{u(D) \cap U} := dx^1 \wedge dx^3 - dx^2 \wedge dx^4.
\end{equation}
This is $J$-anti-invariant. We notice that 
\begin{equation} \label{phi'}
-J\phi_0|_{u(D) \cap U} = dx^1 \wedge dx^4 + dx^2 \wedge dx^3,
\end{equation}
where this latter $J$ now refers to the almost complex structure on $\Lambda_J^-$.

We can extend $\phi_0$ to a section of $\Lambda^-_J$ on $U$. Indeed, by shrinking $U$ if necessary, we may assume that $\Lambda^-_J$ is trivialised over $U$. So we can take a local basis of $\Lambda_J^-$, say $\psi, J \psi$. On $u(D)$ there are functions $h_1,h_2$ such that
\begin{equation*}
\phi_0|_{u(D) \cap U} = h_1 \psi |_{u(D) \cap U}+ h_2 J \psi |_{u(D) \cap U}.
\end{equation*}
Now to extend $\phi_0$ we choose any non-zero smooth extensions of $h_1$ and $h_2$ to $U$.

The foliations described above reduce the study of $Z$ to its intersection with embedded $J$-holomorphic disks. To study such intersections we need to produce an appropriate local trivialisation of $\Lambda_J^-$.

Using the almost complex structure on $\Lambda_J^-$ we can locally choose an orthogonal basis, say, $\phi, J\phi$. We write the $J$-anti-invariant form $\alpha$ locally in terms of this basis, $\alpha=f\phi+gJ\phi$, where $f$ and $g$ are smooth functions.

Lemma \ref{intJdisk2} below establishes a trivialisation for $\Lambda_J^-$ in which $\alpha$ is holomorphic over an embedded $J$-holomorphic disk in terms of the chosen basis. This allows us to establish that if a given embedded $J$-holomorphic disk intersects the zero set non-trivially then the intersection is a finite number of isolated points, and that these intersections are positive. Furthermore these intersections are positive.

\begin{lemma} \label{intJdisk2}
Let $(M,J)$ be an almost complex $4$-manifold and $u:D \rightarrow M$ a smooth, embedded $J$-holomorphic disk. Then for any closed, $J$-anti-invariant $2$-form $\alpha$ there exists a neighbourhood $U \subset M$ of $u(0)$ and a nowhere vanishing $\phi \in \Omega_J^-(U)$ such that for $\alpha$ expressed in terms of the basis $\{\phi, J\phi\}$
\begin{equation} \label{al}
\alpha = f\cdot  \phi + g \cdot J \phi
\end{equation}
 on $U$, the function $(f \circ u) + i (g \circ u)$ is holomorphic  on $u^{-1}(u(D)\cap U)$.
\end{lemma}

We will first write $\alpha$ with respect to the local basis $\phi_0$ and show that the coefficients satisfy a Cauchy-Riemann type equation. From this point an application of the Carleman Similarity Principle allows us to find a local basis whose coefficients are holomorphic. We only state a weak version which is enough for our application. 

\begin{theorem} \label{carleman}
Let $p>2$ and $B_{\varepsilon} \subset \C$ for some $\varepsilon >0$. Suppose that $C_1, C_2 \in L^{\infty}(B_{\varepsilon}, \mathbb C)$ and $v \in W^{1,p}(B_{\varepsilon}, \C)$ is a solution to
\begin{equation}  \label{crsys}
\bar\partial v(z) + C_1(z)v(z) +C_2(z)\bar v(z)=0.
\end{equation}
Then, for a sufficiently small $\delta >0$, there exist functions $\Phi \in C^0(B_{\delta}, \C)$ and $\sigma\in C^{\infty}( B_{\delta}, \C)$ such that $\Phi(z)$ is nowhere zero and
\begin{equation*}
v(z)=\Phi(z) \sigma(z), \quad \bar\partial\sigma=0.
\end{equation*}
\end{theorem}

\begin{remark}\label{!car}
If $C_2=0$ then the transformation $\Phi$ can be found to depend only on $C_1$. 
But in the general case, $\Phi$ will depend on $v$. This implies for different $\alpha$, we would have different diffeomorphisms to realise the trivialisations in Lemma \ref{intJdisk2}. Hence, our argument does not imply the same statement of Lemma \ref{intJdisk2} when we replace $\alpha$ by linear combinations of $\beta$ and $J\gamma$ where  $\beta$ and $\gamma$ are closed $J$-anti-invariant forms. 
This is essentially the reason that our argument would not lead to divisor-to-form correspondence for $J$-anti-invariant forms and their divisors even for tamed $J$. In fact, there are linearly independent $J$-anti-invariant $2$-forms for non-integrable tamed $J$ on a K3 surface \cite{DLZ2}. 
We thank Tedi Draghici for reminding us these examples. 
\end{remark}

The proof of Theorem \ref{carleman} is standard, see {\it e.g.} \cite{ST}. We now use this theorem to prove Lemma \ref{intJdisk2}.

\begin{proof}[Proof of Lemma \ref{intJdisk2}]
Take $\phi_0$ to be the extension of \eqref{phi} described above and write $\alpha = f_0 \phi_0 + g_0 J \phi_0$. Since $\alpha$ is closed, we must have
\begin{equation} \label{clsd}
0=d\alpha = df_0 \wedge \phi_0 + f_0 \, d \phi_0 + dg_0 \wedge J \phi_0 + g_0 \, d(J \phi_0).
\end{equation}

First remark that the subsequent equalities follow from the definition of $\phi_0$,
\begin{align*}
u^*(\partial_3 \inmult \phi_0) &= - ds =  u^*(\partial_4 \inmult (-J \phi_0)), \\
u^*(\partial_4 \inmult \phi_0) &=  dt = u^*(\partial_3 \inmult (J \phi_0)),
\end{align*}
where $z=s+it$ are holomorphic coordinates on $(D, J_0)$ centred at the origin such that $J_0 ds = dt$.

By contracting \eqref{clsd} with $\partial_3$ and pulling back along $u$ we obtain the first of the following expressions of 2-forms on $u^{-1}(U)$. The second is obtained by contracting with $\partial_4$ instead. Using tildes to denote quantities which have been pulled back to $D$ we obtain
\begin{align*}
d\tilde{f}_0 \wedge ds + \tilde{f}_0 \tilde{\beta} - d\tilde{g}_0 \wedge dt - \tilde{g}_0 \tilde{\gamma} = - u^* \left[ \frac{\partial f_0}{\partial x^3} \phi_0 - \frac{\partial g_0}{\partial x^3} J \phi_0 \right] &=0, \\
-d\tilde{f}_0 \wedge dt + \tilde{f}_0 \tilde{\beta}' - d\tilde{g}_0 \wedge ds - \tilde{g}_0 \tilde{\gamma}' = - u^* \left[ \frac{\partial f_0}{\partial x^4} \phi_0 - \frac{\partial g_0}{\partial x^4} J \phi_0 \right]  &=0,
\end{align*}
where $\beta:=\partial_3 \inmult d\phi_0$, $\gamma:=\partial_3 \inmult dJ\phi_0$, $\beta':=\partial_4 \inmult d\phi_0$ and $\gamma':=\partial_4 \inmult dJ\phi_0$. The second equality on each line follows from $u^* \phi_0 = u^* J \phi_0 =0$.

For 1-forms $\eta, \lambda$ on $D$ we have the identity $\eta \wedge J_0\lambda = - J_0\eta \wedge \lambda$. Thus we can rewrite the equations above as,
\begin{align*}
\left( d \tilde{f}_0 + J_0d \tilde{g}_0 \right) \wedge ds =   -\tilde{f}_0 \tilde{\beta} + \tilde{g}_0 \tilde{\gamma}, \\
\left( d \tilde{f}_0 + J_0d \tilde{g}_0 \right) \wedge dt =  \tilde{f}_0 \tilde{\beta}' - \tilde{g}_0 \tilde{\gamma}'. 
\end{align*} 
Or equivalently in terms of components with respect to the coordinates $z=s+it$ on $D$,
\begin{align*}
\frac{\partial \tilde{f}_0}{\partial t} + \frac{\partial \tilde{g}_0}{\partial s}  &=   -\tilde{f}_0 \tilde{\beta}_{12} + \tilde{g}_0 \tilde{\gamma}_{12} \\
\frac{\partial \tilde{f}_0}{\partial s}  - \frac{\partial \tilde{g}_0}{\partial t}  &=   \tilde{f}_0 \tilde{\beta}'_{12} - \tilde{g}_0 \tilde{\gamma}'_{12}.
\end{align*}
This is a Cauchy-Riemann type equation for $\tilde{f}_0 + i \tilde{g}_0$.

By Theorem \ref{carleman} there exists a $\delta >0$, a nowhere zero function $\Phi : B_{\delta} \rightarrow \C$ and a holomorphic function $F: B_{\delta} \rightarrow \C$ such that
\begin{equation}\label{bc}
F_0 = \Phi F, 
\end{equation}
where $F_0=\tilde{f}_0 + i \, \tilde{g}_0$. Henceforth we write $F=\tilde{f}+i\, \tilde{g}$ and $\Phi=\Phi_1 + i\, \Phi_2$.

Define $$\phi|_{u(B_{\delta})} := (\Phi_1 \circ u^{-1})\cdot \phi_0 + (\Phi_2 \circ u^{-1})\cdot  J\phi_0$$ and thus $$J\phi|_{u(B_{\delta})}=-(\Phi_2\circ u^{-1})\cdot \phi_0+(\Phi_1\circ u^{-1})\cdot J\phi_0.$$ These are nowhere vanishing $J$-anti-invariant forms on $u(B_{\delta})$. Extending them to a neighbourhood of $u(0)$ in $M$ we can thus write
\begin{equation*}
\alpha = f \phi + g J\phi,
\end{equation*}
for some smooth functions $f,g : M \rightarrow \R$. By restricting to $u(B_{\delta})$ and applying Equation \eqref{bc}, we have $f \circ u + i \, g \circ  u = F$. The conclusion follows since $F$ is holomorphic.
\end{proof}

\begin{remark}
Above we applied Theorem \ref{carleman} to a Cauchy-Riemann equation whose zeroth order term, in the complex form, is not a multiple of $\tilde{f}_0 + i \tilde{g}_0$. 
Thus the basis $\{\phi, J\phi\}$ found in the lemma will depend on $\alpha$ by Remark \ref{!car}.
\end{remark}

The next lemma establishes a unique continuation result for $Z=\alpha^{-1}(0)$. The result is well-known for self-dual harmonic forms \cite{bar}, and alternately it can be regarded as a corollary to Lemma \ref{intJdisk2} (\textit{c.f.} proof of the upcoming Lemma \ref{highnoopen}).

\begin{lemma} \label{noopen}
Suppose that $\alpha$ is a closed, $J$-anti-invariant $2$-form, then if $\alpha \equiv 0$ on some open set in $M$, it must vanish identically on the whole of $M$. 
\end{lemma}

\begin{proof} 
For any Riemannian metric $g$ compatible with $J$, we have $\Lambda_J^- \subset \Lambda_g^+$. In particular, any closed  $J$-anti-invariant $2$-form is a self-dual harmonic form. Hence any non-trivial, closed, $J$-anti-invariant $2$-form cannot vanish on an open subset of $M$. In fact, from \cite{bar} it is known that such zero sets have Hausdorff dimension $\le 2$. 
\end{proof}

\begin{remark} 
Such unique continuation results for closed $J$-anti-invariant forms are known in all dimensions by elliptic PDE methods, \textit{e.g.} \cite{HMT}.
\end{remark}

We now have all of the necessary ingredients to locally estimate the Hausdorff measure of the zero set $Z$ in Proposition \ref{hausdorff}. In particular, Lemma \ref{intJdisk2} serves the role of Lemma 2.2 of \cite{ZintJ}, {\it i.e.} Gromov's positivity of intersections of a $J$-holomorphic disk and a codimension two almost complex submanifold, in the following proof. 

\begin{proof}[Proof of Proposition \ref{hausdorff}]
This proof is almost identical to the the proof of Proposition 2.4 in \cite{ZintJ}.

First we should remark that since $M$ is compact the Hausdorff measure will be independent of the metric we use. Now for any $x \in Z$ we can find a $J$-fiber-diffeomorphism $Q^x$ of a neighbourhood of $x$ in $M$. By compactness we can choose finitely many of these diffeomorphisms, say $Q^{x_i}$, covering $Z$ and such that the disks are all of the same radius. We show that each $Z \cap Q^{x_i}(D \times D)$ has finite 2-dimensional Hausdorff measure.

Pick $x \in Z$ and write $Q$ for $Q^x$. For each $w$ we know that $Q(D_w)$ intersects $Z$ in finitely many points if it is not totally contained in $Z$ by Lemma \ref{intJdisk2}. We claim that there are only finitely many $w \in \bar{D}$ such that $Q(D_w) \subset Z$.

Suppose that this is not the case. Then we may assume without loss of generality that $0$ is an accumulation point of $w$. We now foliate a neighbourhood of $x$ by $J$-holomorphic disks transverse to $Q(D_0)$, whereby producing an open neighbourhood $M$ which is contained in $Z$. Since this contradicts Lemma \ref{noopen} we will then have the claim.

As before take Gaussian coordinates centred at $x$ but now so that $(0,w')$ is identified with $Q(D_0)$. We choose a $J$-fibre-diffeomorphism $Q':D' \times D' \rightarrow \C^2$, where $D'$ denotes the disk in $\C$ of radius $\rho' < \rho$, such that
\begin{itemize}
\item $\forall w' \in D'$, $Q'(D'_{w'})$ is a $J$-holomorphic submanifold containing $(0,w')$,
\item $\forall w' \in D'$, there exists $z$ depending only on $\Omega$ and $J$ such that
\begin{equation*}
|(\xi',w') - Q'(\xi',w')| \leq z \cdot \rho' \cdot |\xi'|,
\end{equation*}
\item $\forall w' \in D'$, the derivatives of order $m$ of $Q'$ are bounded by $z_m \cdot \rho'$, where $z_m$ depends only on $\Omega$ and $J$.
\end{itemize}
So all the disks $Q'(D'_{w'})$ are transverse to $Q(D_0)$. As being transverse is an open condition we have that $Q'(D'_{w'})$ are transverse to $Q(D_w)$ for all $|w|<\varepsilon$. Thus the intersection points of $Q'(D'_{w'})$ and $Z$ are not isolated and so, by Lemma \ref{intJdisk2}, $Q'(D'_{w'}) \subset Z$. So $Q'(D' \times D') \subset Z$ and since $Q'(D' \times D')$ covers an open neighbourhood of $x$ we have the desired contradiction.

Now we claim that $Q$ may be chosen so that none of the $J$-holomorphic disks are contained in $Z$. In fact we show that there are only finitely many complex directions of $T_xM$ such that there are $J$-holomorphic disks tangent to it and contained in $Z$. With this the claim follows by rotating the Gaussian coordinate system we chose initially. 

Suppose that there are infinitely many such directions. Since the directions in $T_xM$ are parametrised by $\mathbb CP^1$ there is at least one accumulative direction $v$. Choose the Gaussian coordinate system so that $Q(D_0)$ is transverse to $v$, and hence $Q(D_w)$ are transverse to $v$ for small $|w|< \varepsilon$. This is a contradiction with Lemma \ref{intJdisk2} and Lemma \ref{noopen} since the intersection numbers of $Q(D_w) \cap Z$ are infinite for $|w| < \varepsilon$.

Hence if we fix $x$ then we can choose a complex direction such that there is no $J$-holomorphic curve in $Z$ tangent to it. By the perturbative nature of $J$-fibre diffeomorphisms we can choose Gaussian coordinates and a $J$-fibre diffeomorphism so that no $Q(D_w)$ is contained in $Z$ for $w$ sufficiently close to $0$.

Finally we are able to estimate the Hausdorff measure of the compact set $Z \cap Q(\bar{D} \times \bar{D})$. First remark that, by shrinking $D$ if necessary, we may assume without loss of generality that the distortion of $Q$ on the domain $2D \times 2D$ is bounded by some constant $C>0$. Also note that, by our choice of $Q$, for each $w \in \bar{D}$ the set $Z \cap Q(\bar{D}_w)$ is a finite set of points.

Define,
\begin{equation*}
g: \bar{D} \rightarrow \N \cup \{0\}, \quad w \mapsto \#( Z \cap Q(\bar{D}_w)).
\end{equation*}
Clearly this is an upper semi-continuous function and hence achieves a maximal value, say $N$, at some point $w \in \bar{D}$.  Since each intersection point contributes positively by Lemma \ref{intJdisk2}, we know $Z\cap Q(\bar D_w)$ contains at most $N$ points for all $w\in \bar D$. By the Vitali covering lemma we can take a finite cover of the compact set $Z \cap Q(\bar{D} \times \bar{D})$ by balls of radius $\varepsilon$ such that $L$ of these balls are disjoint and the union of  $L$ concentric balls with radius dilated by a factor of $3$ cover. By our distortion assumption each $\varepsilon$ ball intersects $Q(2\bar{D}_w)$ in an open set of area bounded above by $\pi C^2 \varepsilon^2$. The coarea formula then yields,
\begin{equation*}
N\pi C^2 \varepsilon^2 \cdot \pi C^2 (2 \rho)^2 > \frac{1}{2}L \pi^2 \varepsilon^4.
\end{equation*}
Hence there is a constant $C'>0$ such that there can be no more than $C' \varepsilon^{-2}$ balls of radius $3 \varepsilon$ covering $Z \cap Q(\bar{D} \times \bar{D})$. This finishes the proof.
\end{proof}

\section{Positive cohomology assignment}\label{secPCA}
In this section, we will finish the proof of Theorem \ref{4ZJhol}. 

Let us recall the notion of positive cohomology assignment, introduced in \cite{T}. We assume $(X, J)$ is an almost complex manifold, and $C\subset X$ is a set. Let $D\subset \mathbb C$ be the standard unit disk. A map $\sigma: D\rightarrow X$ is called {\it admissible} if $C$ intersects the closure of $\sigma(D)$ inside $\sigma(D)$. Next we define the notion of a positive cohomology assignment to $C$, which is extracted from section 6.1(a) of \cite{T}. 

\begin{definition}\label{PCA}
A positive cohomology assignment to the set $C$ is an assignment of an integer, $I(\sigma)$, to each admissible map $\sigma: D\rightarrow X$. Furthermore, the following criteria have to be met: 
\begin{enumerate}
\item If $\sigma: D\rightarrow X\setminus C$, then $I(\sigma)=0$. 

\item If $\sigma_0, \sigma_1: D\rightarrow X$ are admissible and homotopic via an admissible homotopy (a homotopy $h:[0, 1]\times D\rightarrow X$ where $C$ intersects the closure of Image$(h)$ inside Image$(h)$), then $I(\sigma_0)=I(\sigma_1)$.

\item Let $\sigma: D\rightarrow X$ be admissible and let $\theta: D\rightarrow D$ be a proper, degree $k$ map. Then $I(\sigma\circ \theta)=k\cdot I(\sigma)$.

\item Suppose that $\sigma: D\rightarrow X$ is admissible and that $\sigma^{-1}(C)$ is contained in a disjoint union $\cup_iD_i\subset D$ where each $D_i=\theta_i(D)$ with $\theta_i: D\rightarrow D$ being an orientation preserving embedding. Then $I(\sigma)=\sum_iI(\sigma\circ \theta_i)$.

\item If $\sigma: D\rightarrow X$ is admissible and a $J$-holomorphic embedding with $\sigma^{-1}(C)\ne \emptyset$, then $I(\sigma)>0$.
\end{enumerate}
\end{definition}

The following is Proposition 6.1 of \cite{T}, which will be used to prove Theorem \ref{4ZJhol}.

\begin{prop}\label{pcaholo}
Let $(X, J)$ be a $4$-dimensional almost complex manifold and let $C\subset X$ be a closed set with finite $2$-dimensional Hausdorff measure and a positive cohomology assignment. Then $C$ supports a compact $J$-holomorphic $1$-subvariety. 
\end{prop}

Now we would assign an appropriate positive cohomology assignment to the set $Z=\alpha^{-1}(0)$ for admissible maps. To do this it is convenient to understand $J$-anti-invariant $2$-forms as smooth sections of the complex line bundle $\Lambda_J^-$ over $M$. We shall denote such a section associated with $\alpha$ by $\Gamma_{\alpha}: M \rightarrow \Lambda_J^-$.

Let $\sigma: D\rightarrow M$ be an admissible map and $\alpha$ a $J$-anti-invariant $2$-form. We assign an integer $I_{\alpha}(\sigma)$ as follows. Since $\sigma$ is admissible with respect to the zero set $Z=\alpha^{-1}(0)$, the closure of the image of the composition $\Gamma_{\alpha}\circ \sigma(D)$ intersects the compact manifold $M$, viewed as a submanifold of the total space of the bundle $\Lambda_J^-$, inside $\Gamma_{\alpha}\circ \sigma(D)$. In other words, $\Gamma_{\alpha}\circ \sigma: D\rightarrow \Lambda_J^-$ is admissible with respect to $M\subset \Lambda_J^-$. There exists an arbitrarily small perturbation of $\Gamma_{\alpha}\circ \sigma$ which produces a map $\sigma'$, homotopic to $\Gamma_{\alpha}\circ \sigma$ through admissible maps, such that $\sigma'$ is transverse to $M$.  The set $T$ of intersection points of $\sigma'(D)$ with $M$ is a finite set of signed points. We define $I_{\alpha}(\sigma)$ to be the sum of these signs. 

We now check that $I_{\alpha}$ is a positive cohomology assignment when $\alpha$ is a closed $J$-anti-invariant  $2$-form. In particular, the independence of the perturbations we have chosen follows from the assertion (2) of Definition \ref{PCA}.

\begin{prop}\label{icpca}
Suppose $\alpha$ is a non-trivial closed $J$-anti-invariant $2$-form.
The assignment $I_{\alpha}(\sigma)$ to an admissible map $\sigma: D\rightarrow M$ defines a positive cohomology assignment to $Z=\alpha^{-1}(0)$. 
\end{prop}
\begin{proof}
We will check the assertions (1)-(5) of Definition \ref{PCA} in the following.

If $\sigma(D)\cap \alpha^{-1}(0)=\emptyset$, then $\Gamma_{\alpha}\circ \sigma(D)\cap M=\emptyset$, which implies $I_{\alpha}(\sigma)=0$. This is assertion (1).

Showing assertion (2) is equivalent to showing the following. Let $\sigma_t': D\rightarrow \Lambda_J^-$, $t\in [0, 1]$, be admissible maps with respect to $M$. Let $\sigma_0'$ and $\sigma_1'$ intersect $M$ transversely. Then the intersection numbers ({\it i.e.} the corresponding sums of the signed intersection points $T$) $\sigma_0'\cdot M=\sigma_1'\cdot M$.

To show this, we look at the admissible homotopy $\sigma': D\times I\rightarrow \Lambda_J^-$, where $\sigma'(x, t)=\sigma_t'(x)$. Its boundary map $\partial \sigma': S^2\rightarrow \Lambda_J^-$ is homotopic to zero. Hence $\partial\sigma'\cdot M=0$. Since $\sigma_t'$ are admissible, $M$ intersects  $\partial \sigma'$ only at $\sigma_0'(D)$ and $\sigma_1'(D)$. Moreover, $\partial \sigma'$ induces the reverse orientation at $\sigma_1'(D)$. Hence, $\sigma_0'\cdot M-\sigma_1'\cdot M=\partial\sigma'\cdot M=0$. This implies Definition \ref{PCA}(2), {\it i.e.} $I_{\alpha}(\sigma_0)=I_{\alpha}(\sigma_1)$ if $\sigma_0$ and $\sigma_1$ are connected via an admissible homotopy. 

To show assertion (3), we first choose an admissible map $\sigma': D\rightarrow \Lambda_J^-$ (with respect to $M$) transverse to $M$ which is perturbed from $\Gamma_{\alpha}\circ\sigma$. We can also find a small perturbation $\theta'$ of the degree $k$ map $\theta: D\rightarrow D$ such that there is no critical value of $\theta'$ mapping to $M$ by $\sigma'$. Hence the sum of the signs of the intersection points of $\sigma'\circ \theta': D\rightarrow \Lambda_J^-$ is $k$ times that of $\sigma': D\rightarrow \Lambda_J^-$. Since the number $I_{\alpha}$ is independent of the choice of perturbations by assertion (2), we thus have $I_{\alpha}(\sigma\circ \theta)=k\cdot I_{\alpha}(\sigma)$.

For assertion (4), we choose a perturbation $\sigma': D\rightarrow \Lambda_J^-$ of $\Gamma_{\alpha}\circ \sigma$ such that $\sigma'|_{D-\cup_iD_i}=\Gamma_{\alpha}\circ \sigma |_{D-\cup_iD_i}$. Hence $I_{\alpha}(\sigma)=\sum_i I_{\alpha}(\sigma\circ \theta_i)$.

For the last assertion, let $\sigma: D\rightarrow M$ be an admissible embedded $J$-holomorphic disk. For each intersection point $p\in \sigma^{-1}(\sigma(D)\cap Z)$, we can choose a small neighbourhood $D_p\subset D$ such that, for a certain trivialisation of the complex line bundle $\Lambda_J^-$ over an open neighbourhood $U_p\subset M$ containing $\sigma(D_p)$, the composition $\Gamma_{\alpha}\circ \sigma$ is a holomorphic function over $D_p$ by Lemma \ref{intJdisk2}. Hence, if we perturb this holomorphic function to a nearby one, we will get a holomorphic function with simple zeros. Thus, without loss, we can assume $p$ is such a simple zero. At $\Gamma_{\alpha} \circ\sigma(p)$, the tangent space has the following splitting regarding the orientation$$T_{\Gamma_{\alpha} \circ \sigma(p)}\Lambda_J^-=\Lambda_J^-|_{\sigma(p)}\oplus T_{\sigma(p)}(U_p)=\Lambda_J^-|_{\sigma(p)}\oplus\sigma_*(T_pD_p)\oplus T_{\sigma(p)}(U_p)/\sigma_*(T_pD_p).$$ 
Here, the fibre of the bundle $\Lambda_J^-$ is oriented by local basis $(\phi, -J\phi)$ as in Section \ref{HD2}. Since $D_p$ is a $J$-holomorphic disk in $U_p$, the vector space $T_{\sigma(p)}(U_p)/\sigma_*(T_pD_p)$ is a natural complex plane. Hence, the sign associated with the intersection point $\sigma(p)$ is $+1$. This confirms assertion (5).
\end{proof}

The assignment $I_{\alpha}$ satisfies the assertions Definition \ref{PCA} (1)-(4) for any $J$-anti-invariant $2$-form $\alpha$. The assumption that $\alpha$ is a $J$-anti-invariant closed $2$-form is only used to show assertion (5).

Before we complete the proof of Theorem \ref{4ZJhol}, we recall that given a $J$-holomorphic subvariety $\Theta=\{(C_i, m_i)\}$, there is a natural positive cohomology assignment for its support $|\Theta|=\cup C_i$.  Let $C_i=\phi_i(\Sigma_i)$ where each $\Sigma_i$ is a compact connected complex curve and $\phi_i: \Sigma_i\rightarrow M$ is a $J$-holomorphic  embedding off a finite set. When $\sigma: D\rightarrow M$ is admissible, there is an arbitrarily small perturbation, $\sigma'$,  of $\sigma$ which is homotopic to $\sigma$ through admissible maps and is transverse to $\phi_i$. Each fibre product $T_i:=\{(x, y)\in D\times \Sigma_i| \sigma'(x)=\phi_i(y)\}$ is a finite set of signed points of $D\times \Sigma$. We associate a positive weight $m_i$ to each signed point in $T_i$. The weighted sum of these signs in $\cup T_i$ is a positive cohomology assignment, denoted by $IS_{\Theta}$. 

Conversely, once a positive cohomology assignment $I$ is given as in Proposition \ref{pcaholo} and given $C=\cup C_i$, we can associate the positive weight $m_i$ to $C_i$ as $I(\sigma)$ where $\sigma$ is a $J$-holomorphic disk intersecting transversally to $C_i$ at a smooth point. The cohomology assignment $IS_{\Theta}$ for the subvariety $\Theta=\{(C_i, m_i)\}$ obtained in this way is equal to the original $I$. 

We will now prove Theorem \ref{4ZJhol}.

\begin{proof}[Proof of Theorem \ref{4ZJhol}]
By Proposition \ref{hausdorff} the zero set $Z=\alpha^{-1}(0)$ is a closed set with finite $2$-dimensional Hausdorff measure. By Proposition \ref{icpca}, $Z$ can be endowed with a positive cohomology assignment, $I_{\alpha}(\sigma)$, for each admissible map $\sigma: D\rightarrow M$. Hence, by Proposition \ref{pcaholo}, the zero set  $Z=\alpha^{-1}(0)$ supports a $J$-holomorphic $1$-subvariety. Let $\Theta_{\alpha}$ be the $J$-holomorphic $1$-subvariety determined in the manner described above by the cohomology assignment $I_{\alpha}$. 

The assignment $I_{\alpha}(\sigma)$ for an admissible map $\sigma: D\rightarrow M$ could be understood in the following equivalent way. We look at the disk $\sigma(D)\subset M\subset \Lambda_J^-$ and the section $\Gamma_{\alpha}(M)$ inside the total space of the bundle $\Lambda_J^-$. Then we perturb the section $\Gamma_{\alpha}$ to another one $\Gamma_{\alpha'}$ where $\alpha'$ is a $J$-anti-invariant $2$-form, such that $\Gamma_{\alpha'}$ is transverse to $\sigma(D)$. Moreover, we require $\Gamma_{\alpha'}$ is homotopic to $\Gamma_{\alpha}$ through sections $\alpha_t$ such that $\alpha_t^{-1}(0)\cap \partial \sigma=\emptyset$, $\forall t\in [0, 1]$. The set $T'$ of intersection points of $\sigma(D)$ and $\Gamma_{\alpha'}(M)$ is a finite set of signed points.  Suppose $\sigma$ is of degree $k$ onto its image. Then our $I_{\alpha}(\sigma)$ is $k$ times the sum of these signs in $T'$. 

When we choose $\alpha'$ such that $\Gamma_{\alpha'}(M)\pitchfork M$ inside the total space of $\Lambda_J^-$, we know  $\Gamma_{\alpha'}(M)\cap M$ is a smooth submanifold of $M$ representing the Euler class of the bundle $\Lambda_J^-$. By Proposition 4.3 of \cite{ZintJ}, it is the canonical class $K_J$ of the almost complex manifold $(M, J)$. The sign of each point in $T'$ is equal to the one calculated from the intersection of $\sigma(D)$ with $\Gamma_{\alpha'}(M)\cap M$ inside $M$ if we orient the fibre of the bundle $\Lambda_J^-$ by local basis $\{\phi, J\phi\}$ as in Section \ref{HD2}. 

Since any homology class $\xi \in H_2(M, \mathbb Z)$ is representable by an embedded submanifold, the above claim just implies $\xi\cdot [\Theta_{\alpha}]=\iota_*(\xi) \cdot [M]$ as integers. Here $\iota_*(\xi)$ denotes the induced class in the second Borel-Moore homology of the total space of $\Lambda_J^-$ and the latter product is understood as the intersection pairing in Borel-Moore homology. The homology class $[\Theta_{\alpha}]$ is determined by the intersection pairing with all the classes in $H_2(M, \mathbb Z)$. As explained in the previous two paragraphs,  $\xi\cdot [\Theta_{\alpha}]=\xi\cdot [\Theta_{\alpha'}]=\xi \cdot K_J$, $\forall \xi\in H_2(M, \mathbb Z)$. Hence $\Theta_{\alpha}$ is a $J$-holomorphic $1$-subvariety in the canonical class $K_J$.
\end{proof}

The $J$-holomorphic $1$-subvariety $\Theta_{\alpha}$ determined by the positive cohomology assignment $I_{\alpha}$ corresponding to the closed $J$-anti-invariant form $\alpha$ is called the {\it zero divisor} of $\alpha$.

\begin{remark}
Finally, we remark that the zero locus $Z=\alpha^{-1}(0)$ is exactly where $\alpha$ is degenerate. In particular, it implies $\alpha$ is almost K\"ahler on $M\setminus Z$ if $\alpha$ is a closed $J$-anti-invariant $2$-form. It is direct to see from the local expression. For any point $p\in (M, J)$, the tangent space is identified with a $4$-dimensional real vector space along with a complex structure $J_p$. Let $x_1, x_2, y_1, y_2$ be coordinates centred at $p$ such that $J_p dx_1=-dy_1$ and $J_pdx_2=-dy_2$. Now $(\Lambda_J^-)_p$ is spanned by two non-degenerate $2$-forms $$\beta=dx_1\wedge dx_2-dy_1\wedge dy_2, \, \, \, J_p\beta=dx_1\wedge dy_2+dy_1\wedge dx_2.$$ If $\alpha_p=a\beta+bJ_p\beta$ is degenerate, then there exists an $X\in T_pM$ such that $\beta(aX+bJ_pX, \cdot)=0$. Since $\beta$ is non-degenerate, we must have $a=b=0$. 

Since the first Chern class $c_1(M\setminus Z, J)=0$, we know $M\setminus Z$ is an open symplectic Calabi-Yau $4$-manifold when $\alpha$ is a closed $J$-anti-invariant $2$-form. If the almost complex structure $J$ is compatible with (or tamed by) a symplectic form on $M$, we would like to know whether $M\setminus Z$ is a complex symplectic manifold.
\end{remark}

\section{Higher dimensions}

Our argument can be applied to sections of the canonical bundle in higher dimensions. Let $(M,J)$ be a closed connected almost complex $2n$-manifold. As in the four dimensional case there is a natural generalisation of the canonical bundle, namely the bundle of real parts of $(n,0)$ forms. We will denote this bundle by $\Lambda^{n,0}_{\R}$. The space of its sections is denoted by $\Omega^{n, 0}_{\R}$.

The almost complex structure $J$ on $M$ induces a complex line bundle structure on $\Lambda^{n,0}_{\R}$, we still denote the almost complex structure on $\Lambda^{n,0}_{\R}$ by $J$. Indeed, $J$ on $\Lambda^{n,0}_{\R}$ can be described concretely by its action on a section $\beta$ as follows,
\begin{equation*}
J \beta (X_1,X_2, \cdots,X_n) := -\beta(JX_1,X_2, \cdots,X_n).
\end{equation*}

Using the argument given over the previous two sections we are able to prove the following.

\begin{theorem} \label{highdim}
Let $(M,J)$ be a closed, connected almost complex $2n$-manifold and $\alpha$ a non-trivial, closed form in $\Omega^{n,0}_{\R}$. Then the zero set $Z:=\alpha^{-1}(0)$ is a set of finite $(2n-2)$-dimensional Hausdorff measure admitting a positive cohomology assignment.
\end{theorem}

This naturally asks for a generalisation of Proposition \ref{pcaholo} which we phrase as the following question (Question 3.9 in \cite{ZintJ}).

\begin{question} \label{conj}
Let $(M,J)$ be a closed, connected almost complex $2n$-manifold and $C \subset M$ a closed set with finite $(2n-2)$-dimensional Hausdorff measure and admitting a positive cohomology assignment. Does $C$ support a compact $J$-holomorphic subvariety of complex dimension $n-1$?
\end{question}

If the answer to this question is affirmative then Theorem \ref{highdim} would imply that the zero set of a closed form $\alpha$ in $\Omega^{n,0}_{\R}$ supports a $J$-holomorphic $(n-1)$-subvariety in the canonical class. Recall a $J$-holomorphic $k$-subvariety is a finite set of pairs $\{(V_i, m_i), 1\le i\le m\}$, where each $V_i$ is an irreducible $J$-holomorphic $k$-subvariety and each $m_i$ is a positive integer. Here an irreducible $J$-holomorphic $k$-subvariety is the image of a somewhere immersed pseudoholomorphic map $\phi: X\rightarrow M$ from a compact connected smooth almost complex $2k$-manifold $X$. 

The key to the proof of Theorem \ref{highdim} is to establish foliations by $J$-holomorphic disks in higher dimensions. The following result is Lemma 3.10 in \cite{ZintJ}.
\begin{lemma} \label{highJdisk}
Let $\tilde{J}$ be an almost complex structure on $\C^n$ which agrees with the standard complex structure $J_0$ at the origin. Further let $g$ be a Hermitian metric compatible with $\tilde{J}$. Then there exists a constant $\rho_0 >0$ with the following property. Let $0< \rho < \rho_0$ and $D \subset \C$ the disk of radius $\rho$. Then there exists a diffeomorphism $Q: D \times D^{n-1} \rightarrow \C^n$, and constants $L,L_m$ depending only on $g$ and $\tilde{J}$, such that
\begin{itemize}
\item For all $w \in D^{n-1}$, $Q(D_w)$ is a $\tilde{J}$-holomorphic submanifold containing $(0,w)$;
\item For all $w \in D^{n-1}$, $|(\xi, w) - Q(\xi,w)| \leq L \cdot \rho \cdot |\xi|$;
\item For all $w \in D^{n-1}$, the derivatives of order $M$ of $Q$ are bounded by $L_m \cdot \rho$;
\item For each $\kappa \in \mathbb CP^{n-1}$ we can choose $Q$ such that the disk $Q(D_0)$ is tangent at the origin to the line determined by $\kappa$.
\end{itemize}
\end{lemma}

Now given any point $x \in M$ we can find a local Gaussian coordinate chart and hence Lemma \ref{highJdisk} gives a foliation by $J$-holomorphic disks in a neighbourhood of $x$. 

Fix $x \in M$, we can find a neighbourhood $U$ of $x$ and a non-degenerate 2-form $\Omega$ on $U$ such that $J$ is compatible with $\Omega$ in $U$. This pair $(\Omega,J)$ induce an almost Hermitian metric on $U$. Now we can identify a geodesic ball centred at $x$ with a ball in $\R^{2n}$ centred at the origin. Identify $\R^{2n} = \C^n$ such that
\begin{eqnarray*}
\Omega_x = \omega_0 &=& dx^1 \wedge dx^2 + \cdots+ dx^{2n-1} \wedge dx^{2n} \\
&=& \frac{i}{2} \left( dz^0 \wedge d\bar{z}^0 +\cdots + dz^{n-1} \wedge d\bar{z}^{n-1} \right), 
\end{eqnarray*}
with complex coordinates $(z^0, \cdots, z^{n-1})=(x^1,x^2, \cdots,x^{2n-1},x^{2n})$.
So we may as well assume that $J$ is an almost complex structure on $\C^{n}$ which agrees with the standard complex structure $J_0$ at the origin.

Lemma \ref{highJdisk} gives a $J$-fibre diffeomorphism $Q$ and let $(\xi,\zeta,w)$ be the associated coordinates, where $\xi,\zeta \in D$ and $w=(w^1, \cdots, w^{n-2}) \in D^{n-2}$. Since the disks of constant $(\zeta,w)$ are $J$-holomorphic the almost complex structure $J$ must decompose, with respect to the splitting $T(D \times D \times D^{n-2}) = TD \oplus TD \oplus TD^{n-2} = \R^2 \oplus \R^2 \oplus \R^{2n-4}$, as follows:
\begin{equation*} 
J = \left( \begin{array}{ccc}
a & b_1 & c_1 \\
0 & a' & c_2 \\
0 & b_2 & c_3
\end{array} \right) 
\end{equation*}
Here $a,a',b_1 \in\! \R^{2 \times 2}$, $b_2 \in \R^{(2n-4)\times 2}$, $c_1,c_2 \in \! \R^{2 \times (2n-4)}$ and $c_3 \in \! \R^{(2n-4) \times (2n-4)}$ are matrix valued functions on $D^n$ such that the condition $J^2 = - I$ is satisfied. 

We can further choose coordinates $(\xi_1,\zeta_1, w_1)$ such that $u(D)$ is the disk $\{\xi_1=0, w_1=0\}$, at least locally near $x=u(0)$. To see this first remark that by the final part of Lemma \ref{highJdisk} the $J$-fibre diffeomorphism may be chosen so that $Q(D_0)$ intersects $u(D)$ transversally at $u(0)$. The transversality condition facilitates the application of the implicit function theorem to find, after shrinking $D$ if necessary, smooth functions $\tau_0, \cdots,\tau_{n-2}:D \rightarrow \R^2$ such that $\tau_i(0)=0$ and $u(\zeta)=(\tau_0(\zeta),\zeta,\tau_1(\zeta), \cdots,\tau_{n-2}(\zeta))$. By making the change of coordinates $$(\xi_1,\zeta_1,w_1):=(\xi-\tau_0(\zeta),\zeta,w^1-\tau_1(\zeta), \cdots,w^{n-2}-\tau_{n-2}(\zeta)),$$
we ensure that $u(D)$ is described by $\{\xi_1=0,w_1=0\}$ in a neighbourhood of $x$. Thus in the $(\xi_1,\zeta_1,w_1)$ coordinates we must have $b_1=0$ and $b_2=0$ along the disk $u(D)$. Finally we can make a further change of coordinates to $(\xi_2, \zeta_2, w_2)$ so that
\begin{equation*}
a\equiv \left( \begin{array}{cc}
0 & 1 \\
-1 & 0 
\end{array} \right) \hbox{ and }\,\, 
a'|_{u(D)}=  \left( \begin{array}{cc}
0 & 1 \\
-1 & 0 
\end{array} \right).
\end{equation*}

Applying this process to the complex directions determined by the $n-2$ components of $w_1$, that is, choosing $J$-holomorphic disk foliations along the directions of  $w_1$ at $x=u(0)$ and choose $u(D)$ to be in the center as above, we are able to standardize the coordinate at $u(D)$ such that $J|_{u(D)}$ is a $2n\times 2n$ block matrix with $n$ $2\times 2$ matrices $\left( \begin{array}{cc}
0 & 1 \\
-1 & 0 
\end{array} \right)$. 
Henceforth we let $(z^1,\cdots ,z^n)=(x^1, x^2, \cdots , x^{2n-1}, x^{2n})$ denote the coordinates $(\zeta_2, \xi_2, w_2)$ so that $u(D)$ is defined by $z^2=\cdots =z^{n}= 0$.

We continue assuming $u:D \rightarrow M$ is an embedded $J$-holomorphic disk and $U$ is a neighbourhood of $u(0)$ with the coordinates described above.  On $u(D) \cap U$ define 
\begin{equation*} \label{highphi}
\phi_0|_{u(D) \cap U} := \Re \left[ dz^1 \wedge \cdots \wedge dz^n \right] =\Re \left[ (dx^1 + i dx^2) \wedge \cdots \wedge (dx^{2n-1} + i dx^{2n}) \right].
\end{equation*}

We can extend $\phi_0$ to a form in $\Lambda^{n,0}_{\R}(U)$. Indeed, by shrinking $U$ if necessary, we may assume that $\Lambda^{n,0}_{\R}$ is trivialised over $U$. So we can take a local basis of $\Lambda^{n,0}_{\R}$, say $\psi, J \psi$. On $u(D)$ there are functions $h_1,h_2$ such that
\begin{equation*}
\phi_0|_{u(D) \cap U} = h_1 \psi |_{u(D) \cap U}+ h_2 J \psi |_{u(D) \cap U}.
\end{equation*}
Now to extend $\phi_0$ we choose any non-zero smooth extensions of $h_1$ and $h_2$ to $U$.

A straightforward calculation shows that 
\begin{eqnarray*} \label{highJphi}
J\phi_0|_{u(D) \cap U} &=&\Re \left[ idz^1 \wedge ... \wedge dz^n \right] \\
&=& \Re \left[ (-dx^2 + i dx^1) \wedge ... \wedge (dx^{2n-1} + i dx^{2n}) \right].
\end{eqnarray*}

We can establish positivity of intersections of the zero set with embedded $J$-holomorphic disks. With the coordinates described above one may derive some generalised Cauchy-Riemann equations for the coefficients of $\alpha$ as in Lemma \ref{intJdisk2}. Applying Carleman Similarity Principle we obtain the following lemma.

\begin{lemma} \label{highJdisk2}
Let $(M,J)$ be an almost complex $2n$-manifold and $u:D \rightarrow M$ a smooth, embedded $J$-holomorphic disk. Then for any closed form $\alpha$ in $\Omega^{n,0}_{\R}$ there exists a neighbourhood $U \subset M$ of $u(0)$ and a nowhere vanishing form $\phi$ in $\Omega^{n,0}_{\R}(U)$ such that for $\alpha$ expressed in terms of the basis $\{\phi, J\phi\}$
\begin{equation} \label{al}
\alpha = f \phi + g J \phi
\end{equation}
 on $U$, the function $(f \circ u) + i (g \circ u)$ is holomorphic  on $u^{-1}(u(D)\cap U)$. 
\end{lemma}

\begin{proof} Let $\alpha = f_0 \phi_0 + g_0 J\phi_0$ in terms of the basis $\{\phi_0, J\phi_0 \}$. Then closedness implies,
\begin{equation} \label{dalpha}
0=d \alpha = df_0 \wedge \phi_0 + f_0 d \phi_0 + d g_0 \wedge J\phi_0 + g_0 d(J\phi_0).
\end{equation}
Following the similarity principle argument used in Lemma \ref{intJdisk2} it is enough to verify that $f_0 + i g_0$ satisfies a Cauchy-Riemann type equation.

First remark that 
\begin{align*}
\Re \left[ \partial_{z^2} \inmult ... \inmult \partial_{z^n} \inmult (dz^1 \wedge ... \wedge dz^n) \right]  &=- \Re \left[  i\partial_{z^2} \inmult ... \inmult \partial_{z^n} \inmult (idz^1 \wedge ... \wedge dz^n) \right] \\
&= (-1)^{\frac{n(n-1)}{2}} dx^1, \\
\Re \left[ \partial_{z^2} \inmult ... \inmult \partial_{z^n} \inmult (idz^1 \wedge ... \wedge dz^n) \right] &=\Re \left[  i\partial_{z^2} \inmult ... \inmult \partial_{z^n} \inmult (dz^1 \wedge ... \wedge dz^n) \right] \\  
&= -(-1)^{\frac{n(n-1)}{2}} dx^2.
\end{align*}
This allows us to choose a series of contractions that when applied to \eqref{dalpha} yields the following pair of equations on $D$.
\begin{align*}
d\tilde{f}_0 \wedge ds + \tilde{f}_0 \tilde{\beta} - d\tilde{g}_0 \wedge dt + \tilde{g}_0 \tilde{\gamma} = u^*\Psi_1  &=0, \\
-d\tilde{f}_0 \wedge ds + \tilde{f}_0 \tilde{\beta}' - d\tilde{g}_0 \wedge dt + \tilde{g}_0 \tilde{\gamma}' = u^*\Psi_2  &=0,
\end{align*}
where tildes are used to denote a quantity having been pulled back along $u$, the forms $\beta,\beta'$ are contractions of $d\phi_0$, the forms $\gamma,\gamma'$ are contractions of $d(J\phi_0)$ and $\Psi_i$ are error terms which contain no $dx^1 \wedge dx^2$ terms and hence pull back to $0$.

Arguing identically as in the proof of Lemma \ref{intJdisk2} shows the above pair of equations is a Cauchy-Riemann type system for $\tilde{f}_0 + i \tilde{g}_0$ and that the Similarity Principle gives the desired conclusion.
\end{proof}

This lemma allows us to deduce a unique continuation result for closed sections of the canonical bundle. For completeness we include an elementary proof, but the same result can be deduced from unique continuation for solutions of elliptic PDEs.

\begin{lemma} \label{highnoopen}
Suppose that $\alpha$ is a closed form in $\Lambda^{n,0}_{\R}$, then if $\alpha \equiv 0$ on some open set in $M$, it must vanish identically on the whole of $M$. 
\end{lemma}

\begin{proof}
Suppose that $\alpha$ vanishes on an open subset $U\subset M$. We may further assume that $U$ is the largest open subset where $\alpha$ vanishes. By continuity $\alpha$ vanishes on its closure $\bar U$. If $\bar U\ne M$, choose a point $x\in \partial U:=\bar U\setminus U$. Take a neighbourhood $\mathcal N_x$ of $x$ such that there is a $J$-fibre-diffeomorphism $Q: D\times D^{n-1} \rightarrow \mathcal N_x$. We can take $\rho$ small enough such that each disk $Q(D_w)$ intersects $U$. In particular, for each $w\in D^{n-1}$, $Q(D_w)\cap U$ is an open subset in $Q(D_w)$. However, by Lemma \ref{highJdisk2}, we know $\alpha$ vanishes either at isolated points or totally on $Q(D_w)$. This implies $\alpha|_{Q(D_w)}=0$ for all $w\in D^{n-1}$ and thus $\alpha|_{\mathcal N_x}=0$. Hence $U\cup \mathcal N_x\supsetneq U$, which contradicts the choice of $U$. Hence $\alpha$ vanishes on whole $M$.
\end{proof}

Now Theorem \ref{highdim} follows the same argument as Theorem \ref{4ZJhol}. To show that the $(2n-2)$-dimensional Hausdorff measure of $Z$ is finite we follow the argument of Proposition \ref{hausdorff} but using Lemmas \ref{highJdisk} and \ref{highJdisk2} rather than Lemma \ref{intJdisk2}.

\begin{proof}[Proof of Theorem \ref{highdim}]
Viewing $\alpha$ as a smooth section of the canonical bundle, continuity implies that the zero set is compact.

By compactness we can cover $Z$ by finitely many neighbourhoods which admit $J$-fibre-diffeomorphisms as in Lemma \ref{highJdisk}. So it is enough to show that the intersection of $Z$ with each of these neighbourhoods is of finite $(2n-2)$-dimensional Hausdorff measure. 

Following the arguments of Proposition \ref{hausdorff} we can choose the $J$-fibre-diffeomorphisms as follows. Given $x \in Z$ there is a $J$-fibre-diffeomorphism $Q : D \times D^{n-1} \rightarrow M$ such that $Q(0,0)=x$ and no $J$-holomorphic disk $Q(D_w)$ is contained in $Z$. With such a choice Lemma \ref{highJdisk2} implies that, for each $w \in D^{n-1}$, the intersection $Q(D_w) \cap Z$ is a finite set of points.

Further, by shrinking $D$ if necessary, we may assume without loss of generality that the distortion of $Q$ on the domain $2D \times (2D)^{n-1}$ is bounded by some constant $C>0$.

Define,
\begin{equation*}
g: \bar{D} \rightarrow \N \cup \{0\}, \quad \xi \mapsto \#( Z \cap Q(\bar{D}_w)).
\end{equation*}
Clearly this is an upper semi-continuous function and hence achieves a maximal value, say $N$, at some point $\xi \! \in \! \bar{D}$.  Thus by Lemma \ref{highJdisk2}, we know $Z\cap Q(\bar D_w)$ contains at most $N$ points for all $\xi \in \bar D$. By the Vitali covering lemma we can take a finite cover of the compact set $Z \cap Q(\bar{D} \times \bar{D}^{n-1})$ by balls of radius $\varepsilon$ such that $L$ of these balls are disjoint and the union of  $L$ concentric balls with radius dilated by a factor of $3$ cover. By the distortion assumption each $\varepsilon$ ball intersects $Q(2\bar{D}_w)$ in an open set of area bounded above by $\pi C^2 \varepsilon^2$. The coarea formula then yields,
\begin{equation*}
N\pi C^2 \varepsilon^2 \cdot \pi C^{2n-2} (2 \rho)^{2n-2} > L  \omega_{2n} \varepsilon^{2n},
\end{equation*}
where $\omega_{2n}$ is the volume of the unit $2n$-ball. Hence there is a constant $C'>0$ such that $C' \varepsilon^{-(2n-2)}$ balls of radius $3 \varepsilon$ are enough to cover $Z \cap Q(\bar{D} \times \bar{D}^{n-1})$. This finishes the proof that $\mathcal{H}^{2n-2}(Z) < \infty$.

Finally identical to the argument of Proposition \ref{icpca},  we can verify that the assignment $I_{\alpha}$ of Section \ref{secPCA} defines a positive cohomology assignment for $Z$ in the sense of Definition \ref{PCA}. 
\end{proof}
Since the first Chern class of the complex line bundle $\Lambda_{\R}^{n,0}$ is $K_J$ ({\it e.g.} by the same argument as Proposition 4.3 in \cite{ZintJ}), if Question \ref{conj} is answered affirmatively, the Poincar\'e dual of the homology class of the pseudoholomorphic $(n-1)$-subvariety supported on $Z$ is $K_J$. 

\section{$h_J^-$ is a birational invariant of almost complex $4$-manifolds}\label{bir}

In this section we again specialise to closed almost complex four manifolds. The results of \cite{ZintJ} suggest that the right notion of birational morphism between almost complex four manifolds are degree $1$ pseudoholomorphic maps. Thus we say that two closed almost complex four manifolds $M_1$ and $M_2$ are birational if there exist closed almost complex manfiolds $X_1,...,X_{n+1},Y_1,...,Y_n$ such that $M_1=X_1$, $M_2=X_{n+1}$ and there are degree one pseudoholomorphic maps $\phi_i : Y_i \rightarrow X_i$ and $\psi_i : Y_i \rightarrow X_{i+1}$ for all $i=1,...,n$.

We denote by $\Omega^2$ the space of 2-forms on $M$ ($C^{\infty}$-sections of the bundle $\Lambda^2$), $\Omega_J^+$ the space of $J$-invariant 2-forms, {\it etc.} Let also $ \mathcal Z^2$ denote the space of closed 2-forms on $M$ and let $\mathcal Z_J^{\pm} =\mathcal Z^2 \cap \Omega_J^{\pm}$.

The cohomology groups $$H_J^{\pm}(M)=\{ \mathfrak{a} \in H^2(M;\mathbb R) | \exists \; \alpha\in \mathcal Z_J^{\pm} \mbox{ such that } [\alpha] = \mathfrak{a} \}$$ are defined in \cite{LZ}. The dimensions  of  the vector spaces $H_J^{\pm}(M)$ are denoted as $h_J^{\pm}(M)$. It is proven in \cite{DLZ} that $h_J^++h_J^-=b_2$ when $\dim_{\R} M=4$. 

In this section we prove that $h_J^-$ is a birational invariant of almost complex four manifolds. 

\begin{theorem} \label{birational}
Let $\psi : (M_1,J_1) \rightarrow (M_2,J_2)$ be a degree $1$ pseudoholomorphic map between closed, connected almost complex $4$-manifolds. Then $h_{J_1}^-(M_1)=h_{J_2}^-(M_2)$.
\end{theorem}

The basic strategy is similar to the proof of Theorem 5.3 in \cite{CZ}. However, we need a version of Hartogs' extension theorem for closed $J$-anti-invariant forms. This relies on the trivialisation of $\Lambda_J^-$ over embedded $J$-holomorphic disks provided by Lemma \ref{intJdisk2}.

Again it is convenient to to view $J$-anti-invariant $2$-forms as a smooth sections of the complex line bundle $\Lambda_J^-$ over $M$. We shall denote such a section associated with a $J$-anti-invariant $2$-form $\alpha$ by $\Gamma_{\alpha}: M \rightarrow \Lambda_J^-$. By Lemma \ref{intJdisk2} there is a trivialisation of $\Lambda_J^-$ over a given embedded $J$-holomorphic disk $u:D \rightarrow M$ such that $\Gamma_{\alpha} \circ u$ may be viewed as a holomorphic function $\Gamma_{\alpha} \circ u : D \rightarrow \C$ when $\alpha$ is closed. Notice that once a trivialisation has been chosen we abuse notation and ignore the holomorphic projection of $\Lambda_J^- \cong D \times \C$ onto its second factor. We identify the basis $\{\phi, J\phi\}$ in Lemma \ref{intJdisk2} with $1$ and $i$ in $\C$ under the trivialisation.

Before proceeding it is convenient to make some remarks about Lemma \ref{intJdisk2}. First consider $U \subset M$ an open, connected subset, $\alpha$ a closed $J$-anti-invariant $2$-form defined on $U \! \setminus \! \{ p \}$ for some $p \in U$ and $u : D \rightarrow M$ an embedded $J$-holomorphic disk with $u(0)=p$. It follows from the arguments of Lemma \ref{intJdisk2} that, after possibly shrinking $u(D)$, there is a holomorphic structure on $\Lambda_J^-$ over $u(D) \! \setminus \! \{ p \}$ such that $\Gamma_{\alpha} \circ u : D \! \setminus \! \{ 0 \} \rightarrow \Lambda_J^-$ is holomorphic. 

Indeed, by Lemma \ref{intJdisk2}, we can cover $D \! \setminus \! \{ 0 \}$ by subdisks $D_i$ such that $\Lambda_J^-|_{u(D_i)}$ is trivialised with $\Gamma_{\alpha} \circ u:D_i \rightarrow \Lambda_J^-$ a holomorphic section. Furthermore, we assume the zero locus $\alpha^{-1}(0)\cap u(\partial D_i)=\emptyset$. We look at the transition function $\beta+i\gamma$ of the line bundle $\Lambda_J^-|_{u(D_i)}$ for $D_1\cap D_2$ say. The form $\alpha$ could be represented in terms of two bases $$\alpha=f_1\phi_1+g_1J\phi_1=f_2\phi_2+g_2J\phi_2.$$ By computation, $(f_1+ig_1)=(f_2+ig_2)(\beta+i\gamma)$. In other words, writing $h_i=(\Gamma_{\alpha} \circ u)|_{D_i}$ we can write transition functions as $\tau_{ij}=\frac{h_i}{h_j}$ on $D_{ij}=D_i \cap D_j$. Since the $h_i$ are holomorphic, and the transition functions are nowhere zero, we know $\tau_{ij}$ are holomorphic.

This transition data thus defines a holomorphic line bundle structure on $\Lambda_J^-$ over $u(D) \! \setminus \! \{ p \}$ such that $\Gamma_{\alpha} \circ u : D \! \setminus \! \{ 0 \} \rightarrow \Lambda_J^-$ is holomorphic. Furthermore $D \! \setminus \! \{ 0 \}$ is Stein and hence, by Oka's principle, the bundle is isomorphic to $D \! \setminus \! \{ 0 \} \times \C$. This allows one to view $\Gamma_{\alpha} \circ u : D \! \setminus \! \{ 0 \} \rightarrow \C$ as a holomorphic complex valued function. In summary we have found a trivialisation of $\Lambda_J^-$ over $u(D) \! \setminus \! \{ p \}$ such that $\Gamma_{\alpha} \circ u : D \! \setminus \! \{ 0 \} \rightarrow \C$ is a holomorphic function.

Secondly, for $\varepsilon \in (0,1)$, let $u_{\varepsilon} : D \rightarrow M$ be a smooth family of embedded $J$-holomorphic disks. For each $\varepsilon \in (0,1)$ the arguments of \S2.1 provide coordinates $x^i_{\varepsilon}$ such that $J=J_0$ along $u_{\varepsilon}(D)$. Moreover the $x^i_{\varepsilon}$ vary smoothly in $\varepsilon$. Defining $\phi_{0,\varepsilon}$ by \eqref{phi} and following the arguments of Lemma \ref{intJdisk2} we obtain a family of functions $v_{0,\varepsilon}=f_{0,\varepsilon} + i g_{0, \varepsilon}$ satisfying a Cauchy-Riemann type equation $\bar{\partial} v_{0,\varepsilon} + C_1^{\varepsilon} v_{0,\varepsilon} +C_2^{\varepsilon} \bar v_{0,\varepsilon}  =0$, where $v_{0,\varepsilon}$ and $ 
C^{\varepsilon}_1, C^{\varepsilon}_2$ vary smoothly in $\varepsilon$. Hence the resulting family of holomorphic functions $f_{\varepsilon}+ig_{\varepsilon}$ and forms $\phi_{\varepsilon}$ vary smoothly in $\varepsilon$. That is, the trivialisations over each $u_{\varepsilon}(D)$ vary smoothly.

\begin{prop} \label{Hartog}
Let $(M,J)$ be an almost complex $4$-manifold, $U \subset M$ open and $p \in U$. Suppose that $\alpha$ is a closed $J$-anti-invariant $2$-form defined on $U \! \setminus \! \{ p \}$. Then $\alpha$ extends smoothly to $U$.
\end{prop}

\begin{proof}
First, by shrinking $U$ if necessary, we may assume that there is a $J$-fibre diffeomorphism $Q : D \times D \rightarrow U$ centred at $p$ such that $Q(\{ 0 \} \times D)$ and each $Q(D_w)$ is an embedded $J$-holomorphic disk. 

We trivialise $\Lambda_J^-$ with respect to $\alpha$, first along $Q(\{ 0 \} \times D) \! \setminus \! \{ p \}$ then along each $Q(D_w)$ and $Q(D_0) \! \setminus \! \{p \}$. By the remarks preceding the proposition $\Gamma_{\alpha}$ may be considered a smooth map $\Gamma_{\alpha} : (D \times D) \!\setminus \! \{ (0,0) \} \rightarrow \C$ such that
\begin{itemize}
\item[(i)] $\Gamma_{\alpha}(\cdot, w) : D \rightarrow \C$ is holomorphic for each $w \neq 0$,
\item[(ii)] $\Gamma_{\alpha}(\cdot, 0) : D \! \setminus \! \{ 0 \} \rightarrow \C$ is holomorphic,
\item[(iii)] $\Gamma_{\alpha}(0,\cdot) : D \! \setminus \! \{ 0 \} \rightarrow \C$ is holomorphic.
\end{itemize}

For each $j \in \mathbb{Z}$ define, $$a_j(w) := \int_{|\xi|=\rho} \frac{\Gamma_{\alpha}(\xi,w)}{\xi^{j+1}} \, d\xi.$$ Clearly this is a smooth function $a_j : D \rightarrow \C$ for all $j \in \mathbb{Z}$. Moreover, by (i), we have $a_0(w)=\Gamma_{\alpha}(0,w), w\ne 0$, and hence $a_0 : D \! \setminus \! \{ 0 \} \rightarrow \C$ is holomorphic.

For each $w \neq 0$ the Cauchy Integral formula gives the following Laurent series $$\Gamma_{\alpha}(\xi,w) = \sum_{j=-\infty}^{\infty} a_j(w) \xi^j = \sum_{j=0}^{\infty} a_j(w) \xi^j,$$ where the second equality follows from (i). In particular $a_j(w)=0$ for all $j<0$ and $w \neq 0$. By smoothness of $\alpha$ on $U\setminus \{p\}$ and the trivialisations along the disks, it follows that $a_j(0)=0$ for all $j<0$. Applying Cauchy Integral formula again yields, $$\Gamma_{\alpha}(\xi,0) = \sum_{j=-\infty}^{\infty} a_j(0) \xi^j = \sum_{j=0}^{\infty} a_j(0) \xi^j,$$ proving that $\Gamma_{\alpha}(\xi,0)$ is holomorphic on $D$ with $\Gamma_{\alpha}(0,0) = a_0 (0)$.

Let us now verify that $\Gamma_{\alpha}(0,w)$ can be extended to a holomorphic function on $D$ with value $a_j(0)$ at the origin. To this end notice that, by smoothness, $$\frac{\partial }{\partial \bar{w}} \Gamma_{\alpha}(0,w) = \int_{|\xi|=\rho} \frac{\frac{\partial }{\partial \bar{w}}\Gamma_{\alpha}(\xi,w)}{\xi} \, d \xi=0, \quad \forall w \! \in \! D.$$ So $\Gamma_{\alpha}(0,w)$ extends as a holomorphic function to $D$ and $\Gamma_{\alpha}(0,0) = a_0(0)$.

As remarked in Section \ref{HD2} the $J$-fibre diffeomorphism may be chosen such that $Q(\{ 0 \} \times D)$ is a given $J$-holomorphic disk and $Q(D_0)$ is tangent at $p$ to a given complex direction $\kappa \in \mathbb C P ^1$ transverse to $Q(\{ 0 \} \times D)$. Varying $\kappa$ we produce a family of embedded $J$-holomorphic disks whose complex tangent directions cover a neighbourhood of $\kappa$. Moreover, each of these disks is the $D_0$ fibre of a $J$-fibre diffeomorphism. We can choose finitely many such families whose  union covers a neighborhood of $p$, and their tangent directions cover $\mathbb CP^1$. Since $Q(\{ 0 \} \times D)$ is fixed the argument above provides a holomorphic extension in each complex direction $\kappa$ with the same extended value at $p$. For the disks not transverse to the given $J$-holomorphic disk, we choose any other disk in the family to complete the proof. 
\end{proof}

With this Hartogs type extension in hand, we are able to prove Theorem \ref{birational} (which is identical to Theorem \ref{birationalintro}).

\begin{proof}[Proof of Theorem \ref{birational}]
Since $\psi$ is pseudoholomorphic the pullback of $2$-forms along $\psi$ induces a map $$\psi^* : \mathcal Z_{J_2}^-(M_2) \rightarrow \mathcal Z_{J_1}^-(M_1).$$ We claim that this induced map is an isomorphism. If this is the case then  this induces an isomorphism between $H_{J_1}^-(M_1)$ and $H_{J_2}^-(M_2)$ since $\mathcal Z_J^-$ is isomorphic to $H_J^-$ (see {\it e.g.} \cite{DLZ}).

By Proposition 5.9 of \cite{ZintJ} there exits a finite set $Y \subset M_2$ such that $u|_{M_1 \setminus \psi^{-1}(Y)}$ is a diffeomorphism and $\psi^{-1}(y)$ is a pseudoholomorphic subvariety for all $y \in Y$. Thus, given $\alpha \in \mathcal Z_{J_2}^-(M_2)$, it follows that if $\psi^*(\alpha) = 0$ then $\alpha|_{M_2 \setminus Y}=0$ and hence smoothness implies that $\alpha \equiv 0$.

It is left to show that $\psi^* : \mathcal Z_{J_2}^-(M_2) \rightarrow \mathcal Z_{J_1}^-(M_1)$ is surjective. Since $\psi|_{M_1 \setminus \psi^{-1}(Y)}$ is a diffeomorphism we can pull back a given $\tilde{\alpha} \in \mathcal Z_{J_1}^-(M_1)$ to give a $J_2$-anti-invariant form $\alpha := (\psi^{-1})^*(\tilde{\alpha}) \in \mathcal Z_{J_2}^-(M_2 \setminus Y).$ As $Y$ is a finite set Proposition \ref{Hartog} gives an extension to a form $\alpha \in \mathcal Z_{J_2}^-(M_2)$ which concludes the proof.
\end{proof}

\section{Further discussions}
In this section, we discuss two applications. In Section \ref{mul}, we provide a definition of multiplicity of zeros for a continuous function $u:  D^2 \rightarrow \mathbb R^2$ which generalises the multiplicity of zeros of a holomorphic function. In Section \ref{SW}, we will discuss the relation of Theorem \ref{4ZJhol} with Taubes' program of SW=Gr for $4$-manifolds.
\subsection{Multiplicity of zeros for a continuous function $u:  D^2 \rightarrow \mathbb R^2$}\label{mul}
An amusing application is to define the multiplicity of isolated zeros of a continuous function $u:  D^2 \rightarrow \mathbb R^2$ from the open unit disk $D^2$, as a generalisation of the multiplicity of zeros of a holomorphic function. This subsection could also be viewed as some explicit calculations of the intersection number used throughout the paper. 

Consider a trivial bundle $\mathcal O$ over $D^2$ of real rank two. A continuous function $u:  D^2 \rightarrow \mathbb R^2$, $u(x, y)=(f(x, y), g(x, y))$, is called admissible if $\overline{u^{-1}(0)}\cap \partial D^2=\emptyset$. By taking complex coordinate $z=x+iy$ on $D$ and using the standard identification of $\R^2 = \C$ we can write $u(z)=f(z, 
\bar z)+ig(z, \bar z)$, where $f$ and $g$ are real valued functions. It is clear that this definition of multiplicity also works for an admissible function $u: B^n\rightarrow \mathbb R^n$. 

\begin{example}
The function $u(z)=x$ is not admissible. All non-trivial  holomorphic functions are admissible. The function $u(z)=|z|^2$ is admissible. 
\end{example}

For an admissible function $u:  D^2 \rightarrow \mathbb R^2$, we define the sum of multiplicities of zeros inside $D^2$ by perturbation. We perturb $u$ to a smooth function $\tilde u: D^2 \rightarrow \mathbb R^2$ such that the Jacobian of each zero of $\tilde u$ is non-degenerate. It is equivalent to viewing the function $u$ as a map to the total space of the trivial bundle $\mathcal O$, and requiring the perturbed $\tilde u$ to have transverse intersection with the zero section. Then the multiplicity $I(u)$ is the sum of the signs of the Jacobian of each zero of $\tilde u$. The multiplicity $I(u)$ is independent of the choice of the perturbation $\tilde u$. 

\begin{example}
When $u$ is a holomorphic function, $I(u)$ is just the sum of the multiplicities of all the zeros of $u$ inside the unit disk. Each zero contributes positively to the sum.

One may choose a holomorphic perturbation $u'$ such that $u'$ has more zeros than $u$ over $\mathbb R^2$ and  each zero will contribute positively to the index. A generic holomorphic perturbation would have $I(u)$ many zeros inside the unit disk. 

 On the other hand, if $u$ is an anti-holomorphic function, then each zero contributes negatively.
\end{example}

The following is an explicit example that the multiplicity is independent of the perturbation as long as the Jacobian is non-degenerate at any zero point. 
\begin{example}
Let $u(z)=|z|^2$. Then $I(u)=0$. There are many ways of admissible perturbations. For example, if $\tilde u(z)=|z|^2+\varepsilon z$, then it has two zeros $z=0$ and $z=-\bar{\ep}$. The Jacobian matrix has determinants $|\ep|^2$ and $-|\ep|^2$ at $0$ and $-\bar\ep$ respectively. This implies $I(u)=0$. 

We can also calculate it using other perturbations. A natural one is $\tilde{\tilde {u}}(z)=u(z)+c$. When $c>0$, there will be no zeros in $D$, which again implies $I(u)=0$ immediately. When $c<0$, it is not a good perturbation to calculate the multiplicity since the Jacobian is degenerate at the zero set.
\end{example}

In fact, our multiplicity is uniquely defined in a natural sense. 

\begin{prop}\label{muluni}
The multiplicity $I(u)$ is the unique functional satisfying the following five properties:
\begin{itemize}
\item  $I(u)=0$ if $u(a)\ne 0, \forall a\in D$;

\item If $u_0, u_1: D\rightarrow \R^2$ are admissible and homotopic via an admissible family $u_t$, then $I(u_0)=I(u_1)$;

\item If $\theta: D\rightarrow D$ is a proper degree $k$ map, then $I(u\circ \theta)=k\cdot I(u)$;

\item  If all the zeros are included in disjoint union $\cup_iD_i \subset D$ where each $D_i=\theta_i(D)$ with embedding $\theta_i: D\rightarrow D$, then $I(u)=\sum I(u\circ \theta_i)$;

\item If $u$ is holomorphic, $I(u)$ is the usual multiplicity of zeros for holomorphic functions.
\end{itemize}
\end{prop}
\begin{proof}
By Proposition 3.3 in \cite{ZintJ} (or Proposition \ref{icpca} in this paper), $I(u)$ satisfies the five properties. To show the uniqueness, we first perturb $u$ to $\tilde u$ such that all the zeros are non-degenerate.  We write the Taylor expansion in terms of $z, \bar z$ at each zero of $\tilde u$. By virtue of the fourth item we can, on a small disk around each zero, use a local linear homotopy from $\tilde{u}$ to the linear term of its Taylor expansion. By choosing the disk to be small, no more zeros would be brought in through this homotopy. The linear term at each zero (without loss, we assume the zero is the original point) can be written as $$(\frac{a+d}{2})z+(\frac{a-d}{2})\bar z+(\frac{c-b}{2})iz+(\frac{c+b}{2})i\bar z,$$ where the Jacobian matrix $\left( \begin{array}{cc}
a & b \\
c & d 
\end{array} \right)$ is non-degenerate. If the determinant is positive, a linear homotopy $$(\frac{a+d}{2})z+t(\frac{a-d}{2})\bar z+(\frac{c-b}{2})iz+t(\frac{c+b}{2})i\bar z$$ would lead to a holomorphic function with $I=1$. Notice, when $t\in [0, 1]$, the Jacobians are all non-degenerate. Similarly, when the determinant is negative, it is homotopic to an anti-holomorphic function. By the third item, an anti-holomorphic has the multiplicity opposite to its holomorphic conjugation. Hence, our multiplicity $I(u)$ is uniquely determined by the classical multiplicity of a holomorphic function and the other four properties. 
\end{proof}

\subsection{Taubes' program of SW=Gr for smooth $4$-manifolds}\label{SW}
When a $4$-manifold $X$ admits a symplectic form, Taubes' SW=Gr theorem says that the Seiberg-Witten invariants are equal to well-defined counts of pseudoholomorphic curves. When there is no symplectic form but keeping $b^+>0$,  if we choose a generic Riemannian metric, then there is a closed self-dual $2$-form which is near-symplectic. Recall a closed $2$-form $\omega$ is called near-symplectic if at each point $x$, either $(\omega\wedge\omega)(x)>0$ or  $\omega(x)=0$ and the derivative $(\nabla \omega)(x): T_xX\rightarrow \Lambda^2T_x^*X$ has rank $3$. The zero set $Z$ of a near-symplectic form is a disjoint union of embedded circles. A near-symplectic form restricted on $X\setminus Z$ is symplectic. Taubes speculated that one might extend SW=Gr to $X\setminus Z$. This symplectic form and the metric define a compatible almost complex structure $J$ on $X\setminus Z$. 

Taubes \cite{Tsd} proves the existence of pseudoholomorphic curves provided the non-vanishing of Seiberg-Witten invariants. Precisely, it says that if $X$ has a non-zero Seiberg-Witten invariant then there exists a $J$-holomorphic subvariety in $X\setminus Z$ homologically bounding $Z$ in the sense that it has intersection number $1$ with every linking $2$-sphere of $Z$. We notice that by Gromov's compactness result for pseudoholomorphic curves with boundary \cite{Ye}, we will have pseudoholomorphic subvarieties in $X\setminus Z$ for an arbitrary almost complex structure on $X\setminus Z$ tamed by $\omega$. 
Recently Gerig \cite{Ge} was able to define Taubes' Gromov invariant in this setting and extend SW=Gr over $\mathbb Z/2\mathbb Z$ to non-symplectic $4$-manifolds.

We would like to remark some relations between the above theory and $J$-anti-invariant $2$-forms. We choose a smooth family of closed $2$-forms $\alpha_t$, $ t\in [0,1]$, such that $\alpha_0$ is a $J$-anti-invariant form and any other $\alpha_t$ is near-symplectic. We can also choose a smooth family of almost complex structures $J_t$ defined on $X\setminus Z_t$ when $t\ne 0$ and on $X$ when $t=0$, such that $J_0=J$ and $J_t$ is tamed by $\alpha_t|_{X\setminus Z_t}$. Suppose the Seiberg-Witten invariant of the canonical class is non-trivial. 
By Taubes' and Gerig's results, there are certain well-understood $J_t$-holomorphic subvarieties bounding $Z_t$ when $t\ne 0$. Under this assumption, we should be able to choose a family of such pseudoholomorphic subvarieties $C_t$ such that $Z_0$ is the limit of them in the Gromov-Hausdorff sense. 

This limiting process is related to the local modification of the boundary circles as in the Luttinger-Simpson-Gompf theorem. 
For example, there are an even number of untwisted zero circles in $Z_t$ since $X$ admits an almost complex structure. We recall the type of the zero circles \cite{Ge} in the following. For any oriented circle in $Z_t$, we denote $z$ to be its unit-length tangent vector. Its normal bundle $N$ is split as $L^+\oplus L^-$ where the quadratic form $N\rightarrow \underline{\R}$, $v\mapsto \langle \iota(z)\nabla_v\omega, v\rangle$ is positive and negative definite respectively. A component of $Z_t$ is called untwisted (resp. twisted) if the line bundle $L^-$ is trivial (resp. non-trivial). It would be very interesting to see how the boundary circles ({\it i.e.} the zero circles $Z_t$), and the pseudoholomorphic subvarieties $C_t$ change as $t$ varies. 
In general, the number of boundary circles will change and  thus the topological type of $C_t$ will  change. For instance, it is possible that two untwisted zero circles come together and die, and at the same time the genus of $C_t$ increases by one. As $t$ goes to zero, the number of boundary circles would decrease to zero and we would finally get a closed pseudoholomorphic subvariety $C_0$.

On the other hand, we speculate that the existence of a non-trivial closed $J$-anti-invariant $2$-form implies the existence of an unbounded sequence $\{r_n\}\subset [1, \infty)$ such that the  perturbed Seiberg-Witten equations (2.9) in \cite{Tsd} has solutions, with the spin$^c$ structure whose positive spinor bundle $S^+=T_J^{2,0}M\oplus \underline{\mathbb C}$ and any self-dual harmonic $2$-form $\omega$. Notice we cannot expect the corresponding Seiberg-Witten invariant to be non-trivial. In fact, as pointed out by a referee, any non-K\"ahler proper elliptic surface without singular and multiple fibers has vanishing Seiberg-Witten invariant (see Example 1 of \cite{Biq}) but they have $p_g>0$.

\end{document}